\documentclass{article}
\usepackage{amssymb,amsfonts,latexsym,wasysym}

\newtheorem{theorem}{Theorem}[section]
\newtheorem{remark}[theorem]{Remark}
\newtheorem{lemma}[theorem]{Lemma}
\newtheorem{fact}[theorem]{Facts}
\newtheorem{prop}[theorem]{Proposition}
\newtheorem{cor}[theorem]{Corollary}
\newtheorem{defi}[theorem]{Definition}

\newcommand{\cf}{L(Q^\mathrm{cf}_{\aleph_0})}
\newcommand{\Tb}{\mathbf{T}}

\newenvironment{proof}{\vspace{-0.25cm}
{\bf Proof}: }{\hfill $\Box$}

\begin{document}
\title{Universal theories and compactly expandable models}
\author{Enrique Casanovas and Saharon Shelah	\thanks{Research partially supported by European Research Council grant 338821. No. 1116 on Shelah's publication list. The first author has been partially funded by the Spanish government grant MTM2014-59178-P. }}
\date{May 5, 2017. Last revised February 23, 2019.}
\maketitle
\begin{abstract} Our aim is to solve a quite old question on the difference between expandability and compact expandability. Toward this, we further investigate the logic of countable cofinality.
\end{abstract}

\section{Introduction}

In this article we solve an open problem on expandability of models by an application of   some new results we obtain  on the logic $\cf$, first-order logic with the additional quantifier $Q^\mathrm{cf}_{\aleph_0}$  of cofinality $\aleph_0$.  The syntax of $\cf$ allows the construction of formulas of the form   $Q^\mathrm{cf}_{\aleph_0}xy\varphi(x,y,\overline{z})$.   The meaning of  this formula in a structure $M$ is given by the rule:  $M\models Q^\mathrm{cf}_{\aleph_0}xy\varphi(x,y,\overline{a})$  if and only if  the relation $\{(b,c): M\models\varphi(b,c,\overline{a})\}$ is a linear ordering of the set $\{c: M\models \exists x\varphi(x,c,\overline{a})\}$  and has cofinality $\aleph_0$.

The second author has  introduced  $\cf$  in~\cite{Sh43} and has proved  that    is a fully compact logic, i.e., a set $\Sigma$ of $\cf$-sentences (of any cardinality) has a model if every finite subset of $\Sigma$ has a model. In the same article it is proved that $\cf$ satisfies the L\"owenheim-Skolem theorem down to $\aleph_1$, in the following particularly strong form combining downward an upward L\"owenheim-Skolem theorems: if $\Sigma$ is a set of $\cf$-sentences having an infinite model, then $\Sigma$ has a model in every cardinal $\kappa\geq \aleph_1,|\Sigma|$.  See theorems 2.5 and 2.6 of~\cite{Sh43} applied to $n=0$, $\mu=1$ and $\lambda_0=\aleph_0$.  Compactness of  $\cf$ will be used here to show that for any cardinal $\kappa= 2^{<\kappa}>\aleph_0$, $\cf$  has a $\kappa$-universal  theory.  The L\"owenheim-Skolem theorem won't be used until the very end, in the proof of Proposition~\ref{propos}.

In section 2 we define some classes of $\cf$-theories, in particular the class $\Tb^\mathbf{ec}_{<\kappa}$ of theories with  vocabulary of cardinality $<\kappa$ which are in some sense an analog of existentially closed models and its extension $\Tb^\mathbf{ab}_{<\kappa}$, the class of theories with vocabulary of cardinality $<\kappa$ which are amalgamation bases.  We prove that if $\kappa= 2^{<\kappa}>\aleph_0$, then any amalgamation base $T_\ast\in\Tb^\mathbf{ab}_{<\kappa}$ can be extended to some existentially closed $T\in\Tb^\mathbf{ec}_{<\kappa^+}$ which is universal over $T_\ast$, meaning  that every consistent extension of $T_\ast$ of cardinality $\leq \kappa$ can be embedded over $T_\ast$ in $T$.  The proof only uses compactness and some basic facts  of the logic $\cf$,   such as  finitary character and possibility of renaming, and they can be easily  generalized to other similar compact logics (see Remark~\ref{amt}).

The notions of expandability and compact expandability were introduced by the first author in~\cite{Cas95} and further discussed in~\cite{Cas96}. They are notions of largeness for models related to first-order theories, similar to resplendency but without the use of parameters. In these articles the existence of a compactly expandable model which is not expandable  is left as an open problem. In Section 3 we solve the problem using the tools developed in Section 2.

To simplify notation, we will only consider relational languages, but the results can be  easily adapted to languages with constants and function symbols. A vocabulary $\tau$ is a set of symbols, predicates in our case.  If $T$ is a theory  (a first-order theory or a $\cf$-theory)  $\tau(T)$  will be the vocabulary of $T$.    $\cf(\tau)$ is the set of $\cf$-sentences of vocabulary $\tau$. Along the whole article $\kappa,\mu$ are infinite cardinal numbers.  Along all  this paper, consistent means finitely satisfiable. Since we are dealing with a compact logic, this is the same thing as being satisfiable.

\section{Universal theories in $\cf$}

In this section we work with sentences and theories of the compact logic $\cf$ of countable cofinality.

\begin{defi} Let $\tau_1,\tau_2$ be vocabularies.
\begin{enumerate}
\item  An \emph{embedding} of $\tau_1$ in $\tau_2$ is  a one-to-one mapping $f:\tau_1\rightarrow\tau_2$ that preserves  arities. It is \emph{over $\tau_0\subseteq \tau_1$} if every symbol of $\tau_0$ remains fixed by $f$. It is an \emph{isomorphism} if it is moreover surjective.  If  $\overline{R}= \langle R_1,\ldots,R_n\rangle $  and  $\overline{S}=\langle S_1,\ldots,S_m\rangle$ are tuples of predicates, we write $\overline{R} \approx \overline{S}$  if  they have the same length $n=m$  and the mapping defined by $R_i\mapsto S_i$ is an isomorphism of vocabularies, i.e.,  it is one-to-one and preserves arities.
\item Any embedding $f:\tau_1\rightarrow \tau_2$  induces a \emph{renaming}, a mapping  from $\cf(\tau_1)$ into $\cf(\tau_2)$ for which we will use the same notation $f$. If  $\sigma\in \cf(\tau_1)$, $f(\sigma)$ is the sentence obtained by substitution of every symbol $R\in\tau_1$ of $\sigma$  by the corresponding symbol $f(R)\in\tau_2$.
\item Assume $\tau_1\subseteq \tau_2$.  The notation $\psi(\overline{R},\overline{S})$ will be used for $\cf(\tau_2)$-sentences with the understanding that   $\overline{R}\subseteq  \tau_1$ and  $\overline{S}\subseteq \tau_2\smallsetminus \tau_1$ are  tuples of predicates  without repetitions and they include all predicates appearing in  the sentence.
\end{enumerate}
\end{defi}

\begin{defi} 
\begin{enumerate}
\item Let $\Tb_{<\kappa}$ be the class of consistent $\cf$-theories $T$ such that  $\tau(T)$ has cardinality $<\kappa$.  
\item  Let $\Tb^\mathbf{c}_{<\kappa}$ be the class of all theories $T\in \Tb_{<\kappa}$ which are complete in $\cf (\tau(T))$, so in particular are closed under conjunction.
 \item Let $\Tb^\mathbf{ab}_{<\kappa}$ be the class of all $T_0\in \Tb^\mathbf{c}_{<\kappa}$ which are  amalgamation bases in the following sense:  if  $\tau_0=\tau(T_0)$  and  $\tau_1,\tau_2$ are vocabularies  with $\tau_0=\tau_1\cap \tau_2$  and  $T_l\in \Tb_{<\kappa}$  for $l=1,2$  are  theories with $\tau(T_l)=\tau_l$  and  $T_0=T_1\cap T_2$,  then $T_1\cup T_2$ is consistent  in $\cf$.
 \item Finally, let $\Tb^\mathbf{ec}_{<\kappa}$ be the class of all theories $T\in \Tb^\mathbf{c}_{<\kappa}$ which are existentially closed in the following sense:  for any vocabulary $\tau_1\supseteq \tau(T)$, for any $\cf$-sentence $\psi(\overline{R},\overline{S})$, where     $\overline{R}\subseteq  \tau(T)$ and  $\overline{S}\subseteq \tau_1\smallsetminus \tau(T)$  are tuples of predicates without repetitions, if $T\cup\{\psi(\overline{R},\overline{S})\}$  is consistent, then $\psi(\overline{R},\overline{S}^\prime)\in T$  for some tuple of predicates $\overline{S}^\prime\subseteq \tau(T)$  such that $\overline{S}\approx\overline{S}^\prime$.
\end{enumerate}
\end{defi}

Existentially closed theories   are  a parallel of existentially closed models. We do not mean, of course,  that they are theories  of existentially closed models.

\begin{lemma}\label{amalg}\begin{enumerate}
\item $\Tb_{<\kappa}\supseteq \Tb^\mathbf{c}_{<\kappa} \supseteq \Tb^\mathbf{ab}_{<\kappa}\supseteq \Tb^\mathbf{ec}_{<\kappa}$.
\item Any  $T\in \Tb_{<\kappa}$ can be extended to a member of $\Tb^\mathbf{c}_{<\kappa}$ with the same vocabulary.
\item Assume  $\tau_i\cap \tau_j= \tau$  for  $i<j<\mu$, $T_i\in \Tb_{<\kappa}$, $T\in \Tb^\mathbf{ab}_{<\kappa}$, $\tau(T_i)= \tau_i$, $\tau(T)= \tau$, and  $T\subseteq T_i$. Then $\bigcup_{i<\mu}T_i$ is consistent.
\item If $\kappa >\aleph_0$  and $T\in \Tb_{<\kappa}$, then it can be extended to a member of $\Tb^\mathbf{ec}_{<\kappa}$.
\item For $l=\mathbf{c}, \mathbf{ab}, \mathbf{ec}$: if $f$ is an isomorphism from $\tau_1$ onto $\tau_2$, and $T\in\Tb_{<\kappa}^l$, then $f(T)\in\Tb_{<\kappa}^l$.
\end{enumerate}
\end{lemma}
\begin{proof}  \emph{1}.  We check that  $ \Tb^\mathbf{ec}_{<\kappa}\subseteq \Tb^\mathbf{ab}_{<\kappa}$.  Let $T_0\in  \Tb^\mathbf{ec}_{<\kappa}$  be of vocabulary  $\tau_0$,  let $\tau_1,\tau_2$ be vocabularies  such that  $\tau_0=\tau_1\cap \tau_2$  and let  $T_l\in \Tb^\mathbf{c}_{<\kappa}$  for $l=1,2$  be corresponding theories with $\tau(T_l)=\tau_l$  and  $T_0=T_1\cap T_2$.   Toward  a contradiction, assume  $\psi(\overline{R},\overline{S})\in T_1$,  $\overline{R}\subseteq \tau_0$, $\overline{S}\subseteq \tau_1\smallsetminus \tau_0$ have no repetitions and $\{\psi(\overline{R},\overline{S})\} \cup T_2$ is inconsistent.  Since $\{\psi(\overline{R},\overline{S})\}\cup T_0$ is consistent and $T_0\in \Tb^\mathbf{ec}_{<\kappa}$, for some tuple of predicates $\overline{S}^\prime\approx \overline{S}$ we have  $\psi(\overline{R},\overline{S}^\prime)\in T_0$ and hence  $\{\psi(\overline{R},\overline{S}^\prime)\}\cup T_2$ is consistent. Since the predicates of $\overline{S}$ do not belong to $\tau_2$, we may rename again the formula, showing the consistency of $\{\psi(\overline{R},\overline{S})\}\cup T_2$, which is against our assumption.

\emph{2}.  By compactness of the logic $\cf$.

\emph{3}.  The amalgamation of finitely many theories with common intersection in  $\Tb^\mathbf{ab}_{<\kappa}$ can be proved by induction. The general case follows by compactness.

\emph{4}. Fix a countable vocabulary $\tau_\infty$ disjoint of $\tau(T)$ and containing for each natural number $n\geq 1$ infinitely many  $n$-ary predicates.  Let $T_0=T$ and $\tau_0= \tau(T)$.  We claim that for some vocabulary   $\tau_1\supseteq \tau_0$  of cardinality $|\tau_0|+\aleph_0$ and disjoint of $\tau_\infty$, there is some  complete  $\cf(\tau_1)$-theory  $T_1\supseteq T_0$ such that for every  $\cf$-sentence $\psi(\overline{R},\overline{S})$, if $T_1\cup\{\psi(\overline{R},\overline{S})\}$ is consistent and $\overline{R}\subseteq \tau_0$ and  $\overline{S}\subseteq  \tau_\infty$, then there is a tuple of predicates $\overline{S}^\prime\subseteq \tau_1$ such that $\overline{S}^\prime\approx \overline{S}$ and   $\psi(\overline{R},\overline{S}^\prime)\in T_1$.  For this purpose, let $\mu=|\tau_0|+\aleph_0$ and   let us enumerate $(\sigma_i: i<\mu)$ all  $\cf(\tau_0\cup\tau_\infty)$-sentences.  We inductively define a continuous  ascending chain $(\Sigma_i:i< \mu)$ of sets of $\cf$-sentences.  Start with $\Sigma_0= T_0$.  If  $\sigma_i=\psi(\overline{R},\overline{S})$  (with $\overline{R}\subseteq \tau_0$ and $\overline{S}\subseteq \tau_\infty$) is consistent with $\Sigma_i$, take a new tuple $\overline{S}^\prime\approx\overline{S}$ of predicates  and put $\Sigma_{i+1}=\Sigma_i\cup\{\psi(\overline{R},\overline{S}^\prime)\}$. Otherwise, $\Sigma_{i+1}=\Sigma_i$. In the limit case take the union. Then  let $\tau_1$ be the vocabulary of $ \bigcup_{i<\mu}\Sigma_i$  and let $T_1$ be a complete $\cf(\tau_1)$-theory extending $\bigcup_{i<\mu}\Sigma_i$.

Iterating, we  define $T_{i+1}$ for $i<\omega$ as a complete extension of $T_i$ in a vocabulary $\tau_{i+1}\supseteq \tau_i=\tau(T_i)$ of cardinality $|\tau_{i+1}|=|\tau_i|+\aleph_0$  which is disjoint with $\tau_\infty$. We require    that for every  $\cf(\tau_i\cup\tau_\infty)$-sentence $\psi(\overline{R},\overline{S})$, if $T_{i+1}\cup\{\psi(\overline{R},\overline{S})\}$ is consistent and $\overline{R}\subseteq \tau_i$ and  $\overline{S}\subseteq  \tau_\infty$, then there is a tuple of predicates $\overline{S}^\prime\subseteq \tau_{i+1}$  such that $\overline{S}^\prime\approx \overline{S}$ and $\psi(\overline{R},\overline{S}^\prime)\in T_{i+1}$.  Then   $T^\prime=\bigcup_{i<\omega}T_i$ has the required properties: $T^\prime\in \Tb^\mathbf{ec}_{<k}$  extends $T$. In fact,  its vocabulary $\tau^\prime= \bigcup_{i<\omega}\tau_i$ verifies $|\tau^\prime|= |\tau(T)|+\aleph_0<\kappa$.  Moreover, if $T^\prime\cup\{\psi(\overline{R},\overline{S})\}$ is consistent,  $\overline{R}\subseteq \tau^\prime$  and $\overline{S}\cap\tau^\prime=\emptyset$, we may assume that $\overline{S}\subseteq \tau_\infty$  and we may fix $i<\omega$ such that $\overline{R}\subseteq \tau_i$. Hence $\psi(\overline{R},\overline{S}^\prime)\in T_{i+1}\subseteq T^\prime$  for some  $\overline{S}^\prime\approx\overline{S}$.

\emph{5}.  Obvious.
\end{proof}

\begin{lemma}\label{lemma}\begin{enumerate}
\item Assume $l=\mathbf{c},\mathbf{ab}$ or $\mathbf{ec}$ and $T_i\in \Tb^l_{<\kappa}$ for every $i<\delta$.  If  $T_i\subseteq T_j$ for all $i<j<\delta$  and $\bigcup_{i<\delta}T_i$ has vocabulary $\tau_\delta$ with     $|\tau_\delta|<\kappa$  (in particular, if $\mathrm{cf}(\delta) <\mathrm{cf}(\kappa)$), then $\bigcup_{i<\delta}T_i\in \Tb^l_{<\kappa}$.
\item Assume $l=\mathbf{c},\mathbf{ab}$ or $\mathbf{ec}$, $T_0\subseteq T_2\in \Tb^l_{<\kappa}$, $\mu<\kappa$  and $|T_0|\leq \mu$.  Then $T_0\subseteq T_1\subseteq T_2$ for some $T_1\in\Tb^l_{<\mu^+}$.
\item If  $\aleph_0<\kappa_1<\kappa_2$  and $|\tau(T)|<\kappa_1$, then for  $l=\mathbf{c},\mathbf{ab}$ or $\mathbf{ec}$: $T\in \Tb^l_{<\kappa_1}$ if and only if  $T\in \Tb^l_{<\kappa_2}$.  
\item Let $|\tau(T)|<\kappa$. Then for  $l=\mathbf{c},\mathbf{ab}$ or $\mathbf{ec}$: $T\in \Tb^l_{<\kappa}$  if and only if for some club $E\subseteq [\tau(T)]^{\leq \aleph_0}$,  for every $\tau\in E$,  $T\cap \cf (\tau)\in \Tb^l_{<\aleph_1}$.
\end{enumerate}
\end{lemma}
\begin{proof} \emph{1}.  Clear, using  compactness of $\cf$.

\emph{2}.  This is obvious for complete theories (case $l=\mathbf{c}$). Let us consider existentially closed theories (case $l=\mathbf{ec}$).  Fix, as above, a countable vocabulary $\tau_\infty$ disjoint of $\tau(T_2)$ and containing for each $n\geq 1$ infinitely many  $n$-ary predicates. Let $(\sigma_i: i<\mu)$ be an enumeration of all $\cf(\tau(T_0)\cup\tau_\infty)$-sentences  $\sigma_i=\psi(\overline{R},\overline{S})$ consistent with $T_0$  with $\overline{R}\subseteq \tau(T_0)$ and $\overline{S}\subseteq \tau_\infty$.  We inductively define a corresponding sequence $(\sigma^\prime_i: i<\mu)$ of sentences $\sigma^\prime_i\in T_2$.  If  $\sigma_i=\psi(\overline{R},\overline{S})$ is consistent with $T_2$, we use the fact that $T_2\in\Tb^\mathbf{ec}_{<\kappa}$  to find some tuple $\overline{S}^\prime\subseteq \tau(T_2)$  such that $\overline{S}^\prime\approx \overline{S}$ and $\psi(\overline{R},\overline{S}^\prime)\in T_2$,  and we put  $\sigma_i^\prime= \psi(\overline{R},\overline{S}^\prime)$.  If  $\sigma_i$ is not consistent with $T_2$  we choose as $\sigma_i^\prime$  some sentence in $T_2$ inconsistent with $\sigma_i$. Let $\tau_0$   be the vocabulary of $\{\sigma_i^\prime: i<\mu\}$  and let $\Sigma_0=T_2\restriction \tau_0$. Notice that $T_0\subseteq \Sigma_0$.  Notice that every $\cf$-sentence $\psi(\overline{R},\overline{S})$   with $\overline{R}\subseteq \tau(T_0)$ and $\overline{S}\subseteq \tau_\infty$ which  is consistent with $\Sigma_0$  is also consistent with $T_2$, and therefore  there is some $\overline{S}^\prime\subseteq \tau_0$ such that $\overline{S}^\prime\approx \overline{S}$ and $\psi(\overline{R},\overline{S}^\prime)\in \Sigma_0$.  Now we iterate this construction $\omega$ times, obtaining an ascending chain of vocabularies  $(\tau_i:i<\omega)$  with $\tau(T_0)\subseteq \tau_0$, $\tau_i\subseteq \tau(T_2)$  and $|\tau_i|= \mu$  together with corresponding theories $\Sigma_i=T_2\restriction \tau_i$  such that  for every $\cf$-sentence $\psi(\overline{R},\overline{S})$  consistent with $\Sigma_{i+1}$ and  with $\overline{R}\subseteq \tau_i$ and $\overline{S}\subseteq \tau_\infty$, there is some $\overline{S}^\prime\subseteq \tau_{i+1}$ such that $\overline{S}^\prime\approx \overline{S}$ and $\psi(\overline{R},\overline{S}^\prime)\in \Sigma_{i+1}$. Then $T_1=\bigcup_{i<\omega}\Sigma_i$  satisfies the requirements.

The case of an amalgamation  base ($l=\mathbf{ab}$) is similar. We extend $T_0$  to  $T_1\in\Tb^\mathbf{c}_{<\mu^+}$  such that $T_1\subseteq T_2$  and every sentence $\psi(\overline{R},\overline{S})$ consistent with $T_1$ with $\overline{R}\subseteq \tau(T_1)$ and $\overline{S}\subseteq \tau_\infty$  is  consistent with $T_2$,  and then one easily checks that $T_1\in\Tb^\mathbf{ab}_{<\mu^+}$.

\emph{3}  is clear and \emph{4} follows from \emph{1} and \emph{2}. 
\end{proof}

\begin{lemma}\label{amalg2} Let $T_0,T_1,T_2\in\Tb^\mathbf{ec}_{<\kappa}$ and assume $T_0\subseteq T_1$  and  the embedding $f:\tau(T_0)\rightarrow \tau(T_2)$ maps $T_0$ into $T_2$. Then there is some $T_3\in \Tb^\mathbf{ec}_{<\kappa}$ such that $T_2\subseteq T_3$ and there is some embedding $g:\tau(T_1)\rightarrow \tau(T_3)$ extending $f$ and mapping $T_1$ into $T_3$.
\end{lemma}
\begin{proof} Extend $f$ to an embedding $g$ with domain $\tau(T_1)$ and such that  $g(\tau(T_1))\cap \tau(T_2)= f(\tau(T_0))$,  define $T^\prime_0$ and $T_1^\prime$ as the images  of $T_0$ and $T_1$ by $g$ respectively,  and apply  items \emph{5} and \emph{1} of Lemma~\ref{amalg}  to   $T_0$    to show that $T^\prime_0\in\Tb^\mathbf{ab}_{<\kappa}$.  Now, obviously,  $T_0^\prime\subseteq T_1^\prime\in \Tb_{<\kappa}$,  $T_0^\prime\subseteq T_2\in \Tb_{<\kappa}$ and hence, by definition of $T_0^\prime\in\Tb^\mathbf{ab}_{<\kappa}$,
    $T^\prime_1\cup T_2$ is consistent. By item \emph{4}  of Lemma~\ref{amalg}, it can be extended to some $T_3\in\Tb^\mathbf{ec}_{<\kappa}$.  Clearly, $g$ maps $T_1$ into $T_3$.
\end{proof}

\begin{defi}  Assume $T_{\ast}\in \Tb_{<\kappa}$, $T\in\Tb^\mathbf{ec}_{<\kappa^+}$, and $T_{\ast}\subseteq T$.    We call  the theory $T$  \emph{$\kappa$-universal over $T_{\ast} $}  if   for every $T^\prime\in\Tb^\mathbf{ec}_{<\kappa^+}$ such that $T_{\ast}\subseteq T^\prime$,  there is some embedding $f:\tau(T^\prime)\rightarrow \tau(T)$ over $\tau_{\ast}=\tau(T_{\ast})$ mapping $T^\prime$ into $T$. 
\end{defi}

\begin{theorem}\label{universal} Assume $T_{\ast}\in\Tb^\mathbf{ab}_{<\kappa}$, i.e., it is an amalgamation base.  If $\kappa= 2^{<\kappa}>\aleph_0$, then there exists some   $\kappa$-universal theory  $T\in\Tb^\mathbf{ec}_{<\kappa^+}$ over $T_{\ast}$.
\end{theorem}
\begin{proof} \emph{Case 1}.  $\kappa$ is regular.

Let $\tau_\infty$ be some vocabulary such that $\tau_{\ast}\subseteq \tau_{\infty}$,  $|\tau_\infty|=\kappa$  and for every $n\geq 1$, $\tau_\infty$ has $\kappa$  many $n$-ary predicates.  Let  $(S_i: i<\kappa)$ be a enumeration of  all theories in $\Tb^\mathbf{ec}_{<\kappa}$  whose vocabulary is a subset of $\tau_\infty$.
Notice that every theory in $\Tb^\mathbf{ec}_{<\kappa}$ is  a renaming  of some $S_i$.

 We are going to construct a continuous ascending chain of $\cf$-theories $(T_i:  i<\kappa)$ extending $T_\ast$ such that for every $i<\kappa$:
\begin{enumerate}
\item $T_i\in\Tb^\mathbf{ec}_{<\kappa^+}$
\item If  $j,l<i$, $S_j\subseteq S_l$ and $f$ is an embedding  of $:\tau(S_j)$ into $\tau(T_i)$ mapping $S_j$ into $T_i$, then  $f$ can be extended to an embedding $g: \tau(S_l)\rightarrow \tau(T_{i+1})$  mapping $S_l$ into $T_{i+1}$.
\end{enumerate} 
  We start  by choosing some  arbitrary initial  theory  $T_0\in\Tb^\mathbf{ec}_{<\kappa}$ with  $T_{\ast}\subseteq T_0$ and  $\tau_0=\tau(T_0)\subseteq \tau_\infty$.  By  Lemma~\ref{lemma}, we may take unions at limit stages. Now we show how to obtain $T_{i+1}$.  Let us consider a particular case of  $j,l<i$,  such that $S_j\subseteq S_l$ and $f:\tau(S_j)\rightarrow \tau(T_i)$, an  embedding mapping  $S_j$ into $T_i$.  By Lemma~\ref{amalg2},  there is some $T_{j,l,f}\in\Tb^\mathbf{ec}_{<\kappa^+}$ extending $T_i$ and some embedding $g:\tau(S_l)\rightarrow \tau(T_{j,l,f})$  that maps  $S_l$ into $T_{j,l,f}$ and extends $f$.  The number of the possible triples $(j,l,f)$ is $\leq \kappa$ and hence, by iteration and taking unions at limits, we obtain $T_{i+1}\in \Tb^\mathbf{ec}_{<\kappa^+}$ as desired.

Now  let $T=\bigcup_{i<\kappa}T_i$.   By Lemma~\ref{lemma},  $T\in\Tb^\mathbf{ec}_{<\kappa^+}$.  

\emph{Claim 1}.  If  $S,S^\prime\in\Tb^\mathbf{ec}_{<\kappa}$, $S\subseteq S^\prime$ and $f:\tau(S)\rightarrow \tau(T)$ is an embedding mapping $S$ into $T$, then there is an embedding $f^\prime:\tau(S^\prime)\rightarrow \tau(T)$  extending $f$ that maps $S^\prime$ into $T$.

\emph{Proof of  Claim 1}. Since for some $l<\kappa$ there is some isomorphism  between $\tau(S_l)$ and $\tau(S^\prime)$  mapping $S_l$ onto $S^\prime$, we may assume that  $S^\prime= S_l$ and $S= S_j$  for some $j,l<\kappa$. Now choose $i<\kappa$ such that $i>j,l$  and   the range of $f$ is contained in $\tau(T_i)$  and notice that $f$ maps $S_j$ into $T_i$.   By construction of $T_{i+1}$, there is some extension $f^\prime: \tau(S_l)\rightarrow \tau(T_{i+1})$  of $f$  mapping  $S_l$ into $T_{i+1}$ and hence into $T$. This proves the claim.

We show now that $T$  is  $\kappa$-universal  over $T_{\ast}$.   Let  $T^\prime\in\Tb^\mathbf{ec}_{<\kappa^+}$  be such that $T_{\ast}\subseteq T^\prime$ and decompose it (using Lemma~\ref{lemma}) as $T^\prime=\bigcup_{i<\kappa}T^\prime_i$  where $(T^\prime_i : i<\kappa)$  is a continuous ascending chain of theories $T^\prime_i\in\Tb^\mathbf{ec}_{<\kappa}$. Since $T_\ast\in\Tb^\mathbf{ab}_{<\kappa}$, we may assume that $T_0\subseteq T^\prime$, and hence, without loss of generality, $T^\prime_0 = T_0$. We claim that there is a continuous ascending chain $(f_i:i<\kappa)$ of embeddings $f_i:\tau(T^\prime_i)\rightarrow \tau(T)$ over $\tau_\ast$   mapping $T^\prime_i$ into $T$.  Notice that in this case  $f=\bigcup_{i<\kappa}f_i$ will be an embedding of $\tau(T^\prime)$ into $\tau(T)$ over $\tau_\ast$ mapping $T^\prime$ into $T$.  We start taking  $f_0$ as the identity in $\tau(T^\prime_0)$ and we take unions at limit stages.   If $f_i:\tau(T_i^\prime)\rightarrow \tau(T)$ has been obtained, then  by  Claim 1 we can extend $f_i$ to $f_{i+1}$  mapping $T^\prime_{i+1}$ into $T$. 

\emph{Case 2}.  $\kappa$ is singular.

  By K\"onig's Theorem on cofinality, $\kappa$ is a strong limit.  Let $\theta=  \mathrm{cf}(\kappa)$  and choose $(\kappa_i: i<\theta)$, an increasing  sequence of cardinal numbers which is cofinal in $\kappa$ and such that  $|\tau_{\ast}|<\kappa_0^+$ and  $2^{\kappa_i}\leq \kappa_{i+1}$  for all $i<\theta$.  For each $i<\theta$ let  $\tau_{i,\infty}$ be some vocabulary such that $\tau_{\ast}\subseteq \tau_{i,\infty}$, $|\tau_{i,\infty}|=\kappa_i$  and for every $n\geq 1$, $\tau_{i,\infty}$ has $\kappa_i$  many $n$-ary predicates.     Let  $(S_{i,j}: j<2^{\kappa_i})$ be a enumeration of  all theories in $\Tb^\mathbf{ec}_{<\kappa_i^+}$  whose vocabulary is a subset of $\tau_{i,\infty}$.
Every theory in $\Tb^\mathbf{ec}_{<\kappa_i^+}$ is  a renaming of some $S_{i,j}$.

 We inductively  construct a continuous ascending chain of $\cf$-theories $(T_i:  i<\theta)$ such that for every $i<\theta$:
\begin{enumerate}
\item $ T_i\in\Tb^\mathbf{ec}_{<\kappa_i^+}$
\item If  $j,l<2^{\kappa_i}$,   $S_{i,j}\subseteq S_{i,l}$,  and  $f:\tau(S_{i,j})\rightarrow \tau(T_i)$ is an embedding mapping $S_{i,j}$ into $T_i$, then  $f$ can be extended to an embedding $g: \tau(S_{i,l})\rightarrow \tau(T_{i+1})$  mapping $S_{i,l}$ into $T_{i+1}$.
\end{enumerate}   We choose $T_0\in\Tb^\mathbf{ec}_{<\kappa_0^+}$ arbitrary  with $T_{\ast}\subseteq T_0$ and $\tau_0=\tau(T_0)\subseteq \tau_{0,\infty}$  and we  take unions at limit stages.   This is possible by Lemma~\ref{lemma}, since for any limit ordinal $\delta<\theta$,  $|\bigcup_{i<\delta}\tau(T_i)|\leq \kappa_\delta$. In order to obtain  $T_{i+1}$, we consider  a particular case of $j,l<2^{\kappa_i}$ such that   $S_{i,j}\subseteq S_{i,l}$  and there is some embedding $f:\tau(S_{i,j})\rightarrow \tau(T_i)$ mapping $S_{i,j}$ into $T_i$.  By Lemma~\ref{amalg2},  there is some $T_{j,l,f}\in\Tb^\mathbf{ec}_{<\kappa}$ extending $T_i$ and some embedding $g:\tau(S_{i,l})\rightarrow \tau(T_{j,l,f})$ that extends $f$ and maps  $S_{i,l}$ into $T_{j,l,f}$.  By Lemma~\ref{lemma} we may find such $T_{j,l,f}$ with the additional property that $|\tau(T_{j,l,f})|<\kappa_i^+$.   Note that the number of possible embeddings $f$  from $\tau(S_{i,j})$ into $ \tau(T_i)$ is  $\leq  \kappa_i^{\kappa_i}\leq \kappa_{i+1}$ and therefore the number of triples $(j,l,f)$   is $\leq \kappa_{i+1}$. We may assume that $\tau(T_{j,l,f})\cap \tau(T_{j^\prime, l^\prime,f^\prime})=\tau(T_i)$ whenever $(j,l,f)\neq (j^\prime,l^\prime,f^\prime)$.  By item 3 of Lemma~\ref{amalg}, we see that the union of all $T_{j,l,f}$ is consistent.  Since it has cardinality $\leq \kappa_{i+1}$, by item 4 of Lemma~\ref{amalg} this union can be extended to  $T_{i+1}\in \Tb^\mathbf{ec}_{<\kappa_{i+1}^+}$ as required. 
  
  Now  let $T=\bigcup_{i<\theta}T_i$.   It is clear that  $T\in\Tb^\mathbf{ec}_{<\kappa^+}$. 
  
  \emph{Claim 2}.  If  $S,S^\prime\in\Tb^\mathbf{ec}_{<\kappa_i}$,   $S\subseteq S^\prime$ and $f:\tau(S)\rightarrow \tau(T_i)$ is an embedding mapping $S$ into $T_i$, then there is an embedding $f^\prime:\tau(S^\prime)\rightarrow \tau(T_{i+1})$  extending $f$ that maps $S^\prime$ into $T_{i+1}$.
  
  \emph{Proof of Claim 2}.  Like in the proof of Claim 1, we may assume that  $S^\prime= S_{i,l}$ and $S= S_{i,j}$  for some $j,l<2^{\kappa_i}$.

We finally check  that also in this case  $T$  is $\kappa$-universal over  $T_{\ast}$.  Let  $T^\prime\in\Tb^\mathbf{ec}_{<\kappa^+}$ and decompose it (using Lemma~\ref{lemma}) as $T^\prime=\bigcup_{i<\theta}T^\prime_i$  where $(T^\prime_i : i<\theta)$  is a continuous ascending chain of theories $T^\prime_i\in\Tb^\mathbf{ec}_{<\kappa_i^+}$.  As before, we may assume that $T_0\subseteq T^\prime$ and hence that $T^\prime_0=T_0$. We claim that there is a continuous ascending chain $(f_i:i<\theta)$ of embeddings $f_i:\tau(T^\prime_i)\rightarrow \tau(T_{i+1})$ over $\tau_0$   mapping $T^\prime_i$ into $T_{i+1}$.  Notice that in this case  $f=\bigcup_{i<\theta}f_i$ will be an embedding of $\tau(T^\prime)$ into $\tau(T)$ over $\tau_\ast$ mapping $T^\prime$ into $T$.  We start by taking  $f_0$ as the identity in $\tau_0$ and we take unions at limit stages.   If $f_i:\tau(T_i^\prime)\rightarrow \tau(T_{i+1})$ has been obtained, we use Claim 2 (applied to $i+1$)  to extend $f_i$ to $f_{i+1}$  mapping $T^\prime_{i+1}$ into $T_{i+2}$. 
\end{proof}
 
The  proof of Theorem~\ref{universal} and the preceeding lemmas use only a few properties of the logic $\cf$. The reader can easily check that, besides compactness and  the possibility of building negations and conjunctions of sentences, we only need to rename symbols  and to  give small upper bounds to the number of sentences in given vocabularies. We are only interested here in applications of the logic $\cf$, but it may be convenient to state the main results of this section in the more general setting of abstract model theory. This is done in the next observation, whose proof is left to the reader.

\begin{remark}\label{amt}  Let $\mathcal{L}$ be a compact logic, in the sense of abstract model theory. Assume $\mathcal{L}$ admits renaming,  is closed under boolean operators,  in every sentence only finitely many symbols occur,  and in every finite vocabulary there are only countably many  (or $\leq \theta$) sentences. Define the classes  $\Tb^{l}_{<\kappa}$ and the notion of $\kappa$-universal theory in an analogous way for $\mathcal{L}$.  If $T_{\ast}\in\Tb^\mathbf{ab}_{<\kappa}$ and $\kappa= 2^{<\kappa}>\aleph_0$  (or $>\theta$), then there exists some   $\kappa$-universal theory  $T\in\Tb^\mathbf{ec}_{<\kappa^+}$ over $T_{\ast}$. And similarly the other claims of this section.
\end{remark}

\section{Compact expandability}

The following definition and the  facts stated subsequently are given for models of countable vocabularies. This is only for simplification purposes, everything can be formulated with full generality with a few modifications.

\begin{defi} Let  $M$ be a model of countable vocabulary $\tau$ and of cardinality $\kappa$.
\begin{enumerate}
\item $M$ is \emph{expandable} if for every vocabulary $\tau^\prime\supseteq \tau$ of cardinality $\leq \kappa$,   if $\Sigma$ is a  first-order set of sentences of vocabulary $\tau^\prime$  consistent with the first-order theory $\mathrm{Th}(M)$ of $M$, then there is some expansion $M^\prime$ of $M$ to $\tau^\prime$ such that $M^\prime\models \Sigma$.
\item Call a set of first-order sentences $\Sigma $  of vocabulary $\tau^\prime\supseteq \tau$   \emph{finitely satisfiable in $M$} if for every finite subset $\Sigma_0\subseteq \Sigma$ there is an expansion of $M$ that satisfies $\Sigma_0$.
\item $M$  is \emph{compactly expandable} if for every vocabulary $\tau^\prime\supseteq \tau$ of cardinality $\leq \kappa$,   if $\Sigma$ is a  first-order set of sentences of vocabulary $\tau^\prime$  finitely satisfiable in  $M$, then there is some expansion $M^\prime$ of $M$ to $\tau^\prime$ such that $M^\prime\models \Sigma$.
\end{enumerate}
\end{defi}

The motivation for studying compactly expandable models came originally from the interest in restricted forms of the compactness theorem for logics with standard part.  These logics were considered by M.~Morley in~\cite{Mor73} as generalizations of $\omega$-logic. In $\omega$-logic  some notions concerning the natural numbers remain fixed, in $M$-logic  the structure $M$ replaces the structure of natural numbers. If $M$ is a model of cardinality $\kappa$, for trivial reasons the $M$-logic  does not  satisfy  the compactness theorem for sets of sentences of cardinality larger than  $\kappa$.   One can prove that the model $M$ is compactly expandable if and only if $M$-logic is $\kappa$-compact, that is, if it satisfies the compactness theorem for sets of sentences of cardinality at most $\kappa$. For more on this see~\cite{Cas95} and~\cite{Cas96}.

For a proof of the following list of facts, see~\cite{Cas95}.

\begin{fact}
\begin{enumerate}
\item Saturated and special models are expandable.
\item Expandable models are compactly expandable.
\item Compactly expandable models are $\omega$-saturated and universal.
\item The countable compactly expandable  models are the countable saturated models.
\item If $T$ is superstable and does not have the finite cover property, then every compactly expandable model of $T$ of cardinality $\geq 2^{\aleph_0}$  is saturated.
\item If $T$ is $\aleph_0$-stable and does not have the finite cover property, then every compactly expandable model of $T$ is saturated.
\item If $T$ is superstable and has the finite cover property, $T$ has compactly expandable models which are not saturated.
\item Every unsuperstable  theory  having a saturated model of cardinality $\kappa>\aleph_0$, has a compactly expandable model of cardinality $\kappa$ which is not saturated. 
\end{enumerate}
\end{fact}

The methods used in~\cite{Cas95}  to obtain compactly expandable models which are not saturated nor special are based on ultrapowers and chains of ultrapowers.  They always provide expandable models. In~\cite{Cas96} several examples of theories are discussed where every compactly expandable model is expandable. The question of whether in some theory there exists a compactly expandable model which is not expandable  was asked in~\cite{Cas95} in general and also particularly for the theory of linear dense orders without endpoints. We will give now an affirmative answer.

We are going to apply the results of the previous section to the particular case of the vocabulary $\tau_<=\{<\}$  and to some complete  $\cf$-theory $T_<$  in this vocabulary. We only require from  $T_<$   to extend the first-order theory DLO  of the dense linear order without endpoints and to  contain  the $\cf$-sentences expressing that  $<$ has cofinality $\aleph_0$  while  $>$ has cofinality larger than $\aleph_0$  and for every point $a$,  both $(\{b: b<a\},<)$ and $(\{b: b>a\},>)$ have cofinality larger than $\aleph_0$. More precisely:

\begin{defi}  \label{T<}Let  $\tau_<=\{<\}$  and let 
 $T_<$ be  the  $\cf(\tau_<)$-theory with    the following axioms:
\begin{enumerate}
\item $\forall xyz(x<y\wedge y<z\rightarrow x<z)\wedge \forall x\,(\neg   x<x)\wedge \forall xy( x<y \vee y<x \vee x=y)$
\item $\forall xy( x<y\rightarrow \exists z( x< z \wedge z<y))\wedge \forall x\exists yz( y<x \wedge x<z)$
\item $Q^{\mathrm{cf}}_{\aleph_0} xy\,  (x<y)   \wedge  \neg Q^{\mathrm{cf}}_{\aleph_0}yx\,( x<y)$
\item $\forall x(\neg Q^{\mathrm{cf}}_{\aleph_0}zy (x<y\wedge y<z)\wedge \neg Q^{\mathrm{cf}}_{\aleph_0}yz (y<z \wedge z<x))$
\end{enumerate} 
\end{defi}

\begin{remark} $T_<$  is consistent and complete.
\end{remark}
\begin{proof}  Let $(A,<_A)$ be an $\aleph_1$-saturated dense linear ordering without endpoints and consider the  lexicographic order of the product $ (\omega,<) \times (A,<_A)$.  All the axioms hold in this ordering.   For completeness, use the quantifier elimination of the first-order theory of linear dense orders without endpoints.
\end{proof}

\begin{remark} \label{ast}There is a countable $T_{\ast}\supseteq T_<$  with vocabulary $\tau_{\ast}=\tau(T_{\ast})\supseteq \tau_<$ such that $T_{\ast}\in \Tb^\mathbf{ec}_{<\aleph_1}$; in particular,  $T_{\ast}$ is an amalgamation base.
\end{remark}
\begin{proof} $T_<\in \Tb_{<\aleph_1}$,   and by   items \emph{4} and \emph{1}  of Lemma~\ref{amalg},  it  can be extended to some $T_{\ast}\in\Tb^\mathbf{ec}_{<\aleph_1}\subseteq \Tb^\mathbf{ab}_{<\aleph_1}$.
\end{proof}

\begin{prop} \label{propos} Let $\kappa= 2^{<\kappa}>\aleph_0$ (for instance $\kappa =\beth_\omega$) and let $T\in\Tb^\mathbf{ec}_{<\kappa^+}$ be  $\kappa$-universal over the theory  $T_{\ast}$ from Remark~\ref{ast} (and  hence extend $T_{\ast}\supseteq T_<$).   Then:
\begin{enumerate}
\item $T$ has a model  of cardinality $\kappa$. 
\item If $M$ is a model of $T$ of cardinality $\kappa$  and $M_<= M\restriction \{<\}$, then
\begin{enumerate}
\item $M_<$ is compactly expandable, and
\item $M_<$ is not expandable.
\end{enumerate}
\end{enumerate} 
\end{prop}
\begin{proof} \emph{1}.  Because $\cf$  satisfies the  L\"owenheim-Skolem theorem down to any uncountable cardinal.

\emph{2 (a)}.  Let  $T^\prime$ be a first-order theory in a vocabulary $\tau^\prime$ containing the symbol $<$ and such that $|\tau^\prime|\leq \kappa$, and assume that  $T^\prime$ is finitely satisfiable in $M_<$.  We can assume that $\tau_{\ast}\cap \tau(T^\prime)=\{<\}$ and then $T_{\ast}\cup T^\prime$ is finitely satisfiable in $M_<$ and can be extended to some $T^{\prime\prime}\in \Tb^\mathbf{ec}_{<\kappa^+}$.  Since $T$ is   $\kappa$-universal over $T_{\ast}$, there is some embedding $f:\tau(T^{\prime\prime})\rightarrow \tau(T)$   over $\tau_<=\{<\}$ (and even over $\tau_{\ast}$)  mapping $T^{\prime\prime}$ into $T$.  Since $M\models T$,   $T^{\prime\prime}$  holds in an expansion of $M$  and therefore in an expansion of $M_<$.

\emph{2 (b)}.  Notice that the first-order theory of $M_<$ is DLO, the theory of dense linear orders without endpoints.  There are models of DLO,  hence elementarily equivalent to $M_<$,   where every open interval is isomorphic to the whole model, e.g., the  ordering of the  real numbers. Borrowing terminology from permutation group theory, they are sometimes called doubly transitive linear orders.  This property  can be expressed adding a new predicate to the language, that is, there is a finite vocabulary $\tau^\prime$ containing $<$ and some first-order theory $T^\prime$ of vocabulary $\tau^\prime$  which is consistent with DLO and  in every model of $T^\prime$ every open interval is order isomorphic to the  whole model.  Since $M_<$ has cofinality $\aleph_0$ and all open intervals have cofinality $>\aleph_0$, it is not doubly transitive and, hence,   no expansion of $M_<$  satisfies $T^\prime$.
\end{proof}

\begin{cor} There are compactly expandable linear dense orderings without endpoints which are not expandable.
\end{cor}
\begin{proof}  By Proposition~\ref{propos}.
\end{proof}

\begin{remark}
 In the proof of Proposition~\ref{propos} we can use a DLO isomorphic to its inverse.  There are different choices for the theory $T_{<}$ of Definition~\ref{T<}. We can specify the cofinality of the reverse order  $>$ to be $\aleph_0$.  In fact, we can even not require density of the order.
\end{remark}

\bibliographystyle{alpha}

\noindent{\sc
Departament de Matem\`atiques i Inform\`atica\\
Universitat de Barcelona}\\
{\tt e.casanovas@ub.edu}\\

\noindent{\sc
Einstein Institute of Mathematics, Edmond J. Safra Campus,
Givat
Ram,\\
 The Hebrew University of Jerusalem\\
and Department of Mathematics,
 Hill Center - Busch Campus, Rutgers,\\ The State University of New Jersey.}\\
{\tt shelah@math.huji.ac.il}

\end{document}